\documentclass{amsart}    
\usepackage{amsthm,enumerate, amsmath, amscd, amssymb}
\usepackage[textwidth=27mm, textsize=small]{todonotes}
\usepackage{ulem}

\usepackage{xcolor}
\usepackage[all]{xy}
\input xypic

\theoremstyle{plain}
\newtheorem{theorem}{Theorem}[section]
\newtheorem{lemma}[theorem]{Lemma}
\newtheorem{proposition}[theorem]{Proposition}

\newtheorem{conjecture}[theorem]{Conjecture}

\theoremstyle{definition}
\newtheorem{definition}[theorem]{Definition}
\newtheorem{remark}[theorem]{Remark}
\newtheorem{rk-def}[theorem]{Remark-Definition}

\newtheorem{example}[theorem]{Example}
\theoremstyle{remark}

\usepackage[mathscr]{eucal}
\usepackage{graphics, graphpap}
\usepackage{array, tabularx, longtable}

\usepackage[top=3.0cm, bottom=3.0cm, left=3.0cm, right=3.0cm]{geometry}

\numberwithin{equation}{section}

\def\Jac{\mathrm{Jac}}
\def\JI{\mathcal{J}}

\def\Fitt{\mathbf{Fitt}}

\title{Mather-Yau's type theorem for higher Nash blowup algebras}

\author{Hong Duc Nguyen}
\address{TIMAS, Thang Long University, \newline \indent Nghiem Xuan Yem, Hanoi, Vietnam.} 
\email{duc.nh@thanglong.edu.vn}
\thanks{}


\begin{document}
\begin{abstract}
%
In this paper, we establish a Mather-Yau theorem for higher Nash blowup algebras, demonstrating that the isomorphism type of the local ring of any hypersurface singularity, defined over an arbitrary field, is fully determined by its higher Nash blowup algebras. The classical Mather-Yau theorem (1982) asserts that for isolated complex hypersurface singularities, the isomorphism type of the local ring is determined by the Tjurina algebra. In positive characteristic, this result was extended by considering the higher Tjurina algebras by Greuel and Pham (2017) under the assumptions of an algebraically closed ground field and isolated singularities. Our work begins by proving the stability of higher Nash blowup algebras under contact equivalence in a very general framework. Specifically, we show that the higher Nash blowup algebras of any system of elements in an analytic or geometric ring remain invariant under contact equivalence. For complex hypersurface singularities, this stability was conjectured by Hussain, Ma, Yau, and Zuo, and was recently verified by Le and Yasuda. Finally, the converse is established using a classical result of Samuel (1956).
\end{abstract}
\maketitle

\section{Introduction}
We study in this paper the higher Nash blowup algebras. The study of  higher Nash blowup algebras was motivaed by the notion of  higher Nash blowup, which was introduced by by Yasuda \cite{Yas07} to study the problem of desingularization of varieties.  They are closely related to the higher Jacobian matrices and Jacobian ideals, which are fundamental objects in the study of singularities of varieties as well as singularities of morphisms \cite{BJNB19,Dua17,DNB22}. 

In their paper \cite{HMYZ23}, Hussain-Ma-Yau-Zuo proposed a conjecture predicting that the higher Nash blowup algebras $\mathcal T_n(f)$ of an isolated complex hypersurface singularity $f$ are contact invariants. They verifed this for plane curve singularities and $n=2$, see \cite[Theorem A]{HMYZ23}. The conjecture was proved recently by Le-Yasuda in \cite{LY24} for (non-isolated) complex hypersurface singularities. The approach in \cite{LY24} demonstrates the conjecture through detailed computations of higher Jacobian ideals, highlighting the difficulty of addressing the problem in full generality. In this paper, we overcome these challenges by developing a more general framework, proving the conjecture for cases where $f$ is a system of elements in an analytic or geometric ring  $R$. More precisely, let $R$ be an analytic or geometric  ring (see Definition \ref{a-g-rings}), and let $f=(f_1,\cdots,f_s)$ be a system of elements in $R$. We prove that the $n$-th Nash blowup algebra $\mathcal T_n(f)$ of $f$ is invariant under contact equivalence (Theorem \ref{main1}). 

Motivated by the Mather-Yau theorem, we explore the inverse of the conjecture, which asks whether the contact equivalence class of $f$ can be uniquely determined by its higher Nash blowup algebras. For hypersurface singularities, we provide an affirmative answer to this question. The classical Mather-Yau theorem \cite{MaY82} establishes that for isolated complex hypersurface singularities, the isomorphism type of its local ring is determined by its Tjurina algebra. However, this result does not hold in positive characteristic as pointed out by Mather-Yau. Greuel-Pham \cite{GrP17} and Kerner \cite{Ker22} addressed this limitation by proving that the isomorphism type of the local ring of an isolated hypersurface singularity over a field of positive characteristic can be determined by sufficiently high Tjurina algebras. The analogous question for non-isolated hypersurface singularities remains unresolved. 

In this paper, we extend this investigation by replacing higher Tjurina algebras with higher Nash blowup algebras. We demonstrate that the isomorphism type of the local ring of a hypersurface singularity $f$ is uniquely determined by its higher Nash blowup algebra of degree at least 2, without imposing any assumptions on $f$ or the ground field (Theorem \ref{main2}).

\section{Higher Jacobian ideals and contact equivalence}

\subsection{Higher Jacobian ideals}\label{Sec2.1}
We recall in this section the notion of higher jacobian ideals of a system of elements of an analytic or geometric $\Bbbk$-algebra. 
\begin{definition}\label{a-g-rings}
A $\Bbbk$-algebra $R$ is called  is called {\it geometric} if it is either of finite type $\Bbbk$-algebra, or a localization or a completion of a finite type $\Bbbk$-algebra. If the field $\Bbbk$ admits a non-trivial valuation, then we denote by $\Bbbk\{\bf x\}$ the ring of convergent power series, where ${\bf x}=(x_1,\ldots,x_d)$. We call $R$ {\it an analytic} ring if it is isomorphic to a quotient of  the ring $\Bbbk\{\bf x\}$.
\end{definition}
From now on, we always assume that $R$ is either analytic or geometric. Let ${\bf t}=(t_1,\ldots,t_s)$ and $S$ be either $\Bbbk[{\bf t}]$ or $\Bbbk\{\bf t\}$ or $\Bbbk[\![\bf t]\!]$. Let $f=(f_1,\ldots,f_s)$ be a system of non-zero divisors element of an analytic or geometric $\Bbbk$-algebra $R$ such that the corresponding $t_i\mapsto f_i$ defines a morphism of $\Bbbk$-algebras. We denote by $R\otimes_f R$ the tensor product over $S$ and consider it as $R$-module via the map
$$\delta_f: R\to R\otimes_f R, \alpha\mapsto \alpha\otimes_f 1.$$ 
Let $\rho_f\colon  R\times_f R\to R$ defined as $\rho_f(\alpha\otimes \beta)=\alpha\beta$ and let $\mathcal I_f$ be the kernel of $\rho_f$. For any $n\in \mathbb N^*$, we define the {\it $R$-module of K\"ahler differentials of order $n$} of $f$ by
$$\Omega_f^{(n)}:=\mathcal I_f/_{\mathcal I_f^{n+1}}.$$ 
\begin{lemma}
Assume that $R$ is analytic or geometric. Then the $R$-module $\Omega_f^{(n)}$ is finite.
\end{lemma}
\begin{proof}
We prove for the geometric case since the analytic case is similar. Assume that $R$ is a localization or a completion of the ring $\Bbbk[{\bf x}]/I$. Then it is easily seen that
$$\mathcal I_f=\left((x_i\otimes 1-1\otimes x_i\mid i=1,\ldots,d\right)$$
and hence  the $R$-module $\Omega_f^{(n)}$ is generated by the elements $({\bf x}\otimes 1-1\otimes {\bf x})^\alpha$ with $1\leq |\alpha|\leq n$, where
$$({\bf x}\otimes 1-1\otimes {\bf x})^\alpha:=(x_1\otimes 1-1\otimes x_1)^{\alpha_1}\ldots (x_d\otimes 1-1\otimes x_d)^{\alpha_d} \text{ and } |\alpha|=\alpha_1+\ldots+\alpha_d.$$ 
\end{proof}
Let ${M}$ be any finite ${R}$-module. Choose a presentation
$$
\bigoplus_{j \in J} R \longrightarrow R^{\oplus n} \longrightarrow M \longrightarrow 0
$$
of $M$. Let $A=\left(a_{i j}\right)_{i=1, \ldots, n, j \in J}$ be the matrix of the map $\bigoplus_{j \in J} R \rightarrow R^{\oplus n}$. The $k$-th Fitting ideal of $M$ is the ideal $\Fitt_k(M)$ generated by the $(n-k) \times(n-k)$ minors of $A$, which is independent of the choice of the presentation.
\begin{definition}\label{higherjacobian}
Let $R$ be an analytic or geometric $\Bbbk$-algebra and let $e=\dim R/_{(f)}, r={{n+e}\choose{e}}-1$, where $(f)$ is the ideal generated by $f_1,\ldots,f_s$. Then the $r$-th Fitting ideal $\Fitt_r(\Omega_f^{(n)})$ of the $R$-module $\Omega_f^{(n)}$ is called the {\it $n$-th Jacobian ideal of $f$} and denoted by $\JI_n(f)$.	
\end{definition}

Let us give an explicit description of $\JI_n(f)$ in the case when $f$ is an element of 
$R=\Bbbk[\mathbf{x}]$ or $\Bbbk\{\mathbf{x}\}$ or $\Bbbk[\![\mathbf{x}]\!]$, where $\mathbf{x}=(x_1,\dots,x_d)$. To an $f\in R$ and an $n\in \mathbb N^*$ we associate a matrix described as follows
$$\Jac_n(f):=\left(r_{\beta,\alpha}\right)_{0\leq |\beta|\leq n-1, 1\leq |\alpha|\leq n},$$
where the ordering for row and column indices is graded lexicographical, 
\begin{equation}\label{matrixJac}
r_{\beta,\alpha}=r_{\beta,\alpha}(f):=
\begin{cases}
0 & \text{if}\ \ \alpha_i<\beta_i \ \ \text{for some}\ 1\leq i\leq d\\
0 & \text{if} \ \ \alpha=\beta\\
\frac{\partial^{\alpha-\beta}f}{(\alpha-\beta)!} & \text{if}\ \ \alpha> \beta,
\end{cases}
\end{equation}
where $\alpha> \beta$ if $\alpha\neq \beta$ and $\alpha_i\geq\beta_i$ for all $i$. Clearly, $\Jac_n(f)$ is a matrix of type ${{d-1+n}\choose{d}}\times \big({{d+n}\choose{d}}-1\big)$ with entries in $R$.

\begin{definition}\label{bdef}
Let $R=\Bbbk[\mathbf{x}]$, $\Bbbk\{\mathbf{x}\}$ or $\Bbbk[\![\mathbf{x}]\!]$. For $f\in R$, the matrix $\Jac_n(f)$ is called the {\it Jacobian matrix of order $n$ of $f$}, or the {\it $n$-th Jacobian matrix of $f$}. 
\end{definition} 

\begin{remark}\label{versions-Jac}
There are a few versions of higher Jacobian matrix which are slightly different one another. Our definition above follows the one adopted in \cite{BD20}. Another version considered in \cite{Dua17,HMYZ23} differs in that  the diagonal entries \(r_{\alpha,\alpha}\) are \(f\) instead of \(0\). These two versions coincide modulo \(f\).  The one in \cite{BJNB19}, which the authors call the {\it Jacobi-Taylor matrix}, has one extra column by allowing \(|\alpha|=0\). 
\end{remark}

\begin{proposition}\cite[Proposition 2.15]{LY24}\label{relmatrix}
Let $R=\Bbbk[\mathbf{x}]$, $\Bbbk\{\mathbf{x}\}$ or $\Bbbk[\![\mathbf{x}]\!]$. The $n$-th Jacobian ideal $\JI_n(f)$ of $f$  is generated by all the maximal minors of the matrix $\Jac_n(f)$.
\end{proposition}

\begin{proposition}\cite[Proposition 2.19]{LY24}\label{inclusion}
Let $R=\Bbbk[\mathbf{x}]$, $\Bbbk\{\mathbf{x}\}$ or $\Bbbk[\![\mathbf{x}]\!]$. We have the inclusions
$$\mathcal{J}_n(f)\subseteq \mathcal{J}_1(f){ }^{\binom{d-2+n}{d-1}}.$$
 In particular, $\mathcal{J}_n(f) \subseteq\mathcal{J}_1(f)^3$, if either
\begin{itemize}
\item $d \geq 3$ and $n \geq 2$, or
\item $ d=2$ and $n \geq 3$.
\end{itemize}
\end{proposition}

\subsection{Higher Nash blowup algebras and contact equivalence}\label{Sec2.2}
Let $f$ be a systerm of regular elements in $R$. The $n$-th Jacobian  and Nash blowup local algebra of $f$ is defined respectively as
\begin{align*}
\mathcal M_n(f) &:=R/_{\JI_n(f)}\\
\mathcal T_n(f)&:= R/_{(f)+ \JI_n(f)}.
\end{align*} 
In the following we will show that these algebras are stable under right and contact equivalence respectively. Recall that two systems $f=(f_1,\ldots,f_s)$ and $g=(g_1,\ldots,g_s)$ of elements in $R$ are called {\it right equivalent}, denoted by $f\sim_r g$, if there is an automorphism $\varphi$ of $\Bbbk$-algebra of $R$ such that $g_i=\varphi(f_i)$. They are called {\it contact equivalent}, denoted by $f\sim_c g$, if there are in addition units $u_1,\ldots,u_s\in R$ such that $g_i=u_i\varphi(f_i)$. 

In their paper \cite{HMYZ23}, Hussain-Ma-Yau-Zuo proposed the following conjecture, and proved its for $d=n=2$, see \cite[Theorem A]{HMYZ23}.
\begin{conjecture}[\cite{HMYZ23}, Conjecture 1.5]\label{conjj}
Let $f$ and $g$ be in $\mathbb C\{{\bf x}\}$ with $f(\mathbf 0)=g(\mathbf 0)=0$. If $f$ is contact equivalent to $g$ at $\mathbf 0$, then $\mathbb C\{{\bf x}\}/_{(f)+\JI_n(f)}$ is isomorphic to $\mathbb C\{{\bf x}\}/_{(g)+\JI_n(g)}$ as $\mathbb C$-algebras for any $n\in \mathbb N^*$.	
\end{conjecture}
This conjecture was recently proved by Le-Yasuda in \cite[Theorem 2.26]{LY24} without the assumption that $f$ and $g$ have an isolated singularity at the origin. In the following we prove the conjecture in very general sense. Namely, the ground field is arbitrary, the ring $R$ is analytic or geometric (see Section \ref{Sec2.1}), and $f$ and $g$ are two systems of elements in $R$, not just hypersurfaces. The converse is shown to hold true for hypersurface singularities in the next section (Theorem \ref{main2}).

\begin{theorem}\label{main1}
Assume that $R$ is analytic or geometric. Let $f,g$ be two systems of $s$ regular elements in $R$. 
\begin{itemize}
	\item[(i)] If $f$ and $g$ are right equivalent, then $\mathcal M_n(f)\cong \mathcal M_n(g).$	
	\item[(ii)] If $f$ and $g$ are contact equivalent, then $\mathcal T_n(f)\cong \mathcal T_n(g).$
\end{itemize} 
\end{theorem}
\begin{proof}
The theorem follows directly from Lemmas \ref{local-lem1} and \ref{local-lem3} below.
\end{proof}
\begin{lemma}\label{local-lem1}
Let $\varphi$ be an automorphism of $R$. Then for $n\in \mathbb N^*$,  $\varphi(\JI_n(f))=\JI_n(\varphi(f))$.
\end{lemma}

\begin{proof}
We denote $\varphi\colon R\to R'=R$ and  $g=\varphi(f)$ in $R'$. Let  
$$\varphi\otimes\varphi\colon  R\otimes_f R\to R'\otimes_g R'$$
denote the isomorphism mapping $\alpha\otimes\beta$ to $\varphi(\alpha)\otimes\varphi(\beta)$. The morphisms $\varphi\otimes\varphi$ and $\delta_g$ give rise to an isomorphism $\phi$ and the following commutative diagram
\begin{displaymath}
\xymatrix@=3 em{
	R  \ar[d]_{\delta_f}\ar[r]^{\varphi}& R'\ar[d]^{}\ar[ddr]^{\delta_g}&\\
	R\otimes_f R \ar[r]_{}\ar[drr]_{\varphi\otimes\varphi}& (R\otimes_f R)\otimes_R R'\ar[dr]^{\phi}&\\
&&R'\otimes_g R'.
}
\end{displaymath}
 We obtain the following isomorphisms of $R'$-modules induced by $\phi$  
$$\mathcal I_f\otimes_R R'\cong \mathcal I_g$$
and therefore
$$\Omega^{(n)}_f\otimes_R R'\cong \Omega^{(n)}_g.$$
Since taking Fitting ideals commutes with base change, it yields that
$$\varphi\left(\Fitt_r(\Omega_f^{(n)})\right)=\Fitt_r(\Omega_f^{(n)}\otimes_R R')=\Fitt_r(\Omega_g^{(n)})$$
as ideals in $R'=R$. Hence the lemma follows.
\end{proof}
\begin{lemma}\label{local-lem3}
Let $g=uf$ where $u=(u_1,\ldots,u_s)$ is a system of unit elements in $R$. Then for any $n\in \mathbb N^*$,  
$$(f)+ \JI_n(f)=(g)+\JI_n(g).$$
\end{lemma}

\begin{proof}
Let $$\varphi\colon R\to R':=R/_{(f)}=R/_{(g)}$$ denote the natural projection, which maps $f$ and $g$ to $0$. Note that 
$$ R\otimes_f R\cong  R\otimes_\Bbbk R/_{(f\otimes 1-1\otimes f)}, \alpha\otimes\beta\mapsto \alpha\otimes\beta.$$
and hence 
$$ (R\otimes_f R)\otimes_R R'\cong  R\otimes_\Bbbk R/_{(f\otimes 1,1\otimes f)}.$$
Similarly one has
$$ (R\otimes_g R)\otimes_R R'\cong  R\otimes_\Bbbk R/_{(g\otimes 1,1\otimes g)}.$$
Since $g=uf$, it follows that $$R\otimes_\Bbbk R/_{(f\otimes 1,1\otimes f)}=R\otimes_\Bbbk R/_{(g\otimes 1,1\otimes g)}.$$
We obtain an isomorphism 
$$ \phi\colon(R\otimes_f R)\otimes_R R'\cong  (R\otimes_g R)\otimes_R R'$$
which maps $\left((\alpha\otimes\beta)\otimes\gamma\right)$ to $\left((\alpha\otimes\beta)\otimes\gamma\right)$. The morphism $\phi$ induces an isomorphism of $R'$-modules 
$$\phi\colon \Omega^{(n)}_f\otimes_R R' \to  \Omega^{(n)}_g\otimes_R R'.$$
It yields that, the following identities hold in $R'$
\begin{align*}
\varphi(\JI_n(f))&=\varphi\left(\Fitt_r(\Omega_f^{(n)})\right)\;\;\;=\Fitt_r(\Omega_f^{(n)}\otimes_R R')\\
&=\Fitt_r(\Omega_g^{(n)}\otimes_R R')=\varphi(\JI_n(g)).
\end{align*}
Hence the identity
$$(f)+ \JI_n(f)=(g)+\JI_n(g)$$
holds in $R$.
\end{proof}

\section{Mather-Yau theorem for higher Nash blowup algebras}
The well-known Mather-Yau theorem \cite{MaY82} says that the isomorphism type of the local ring of an isolated complex hypersurface singularity is determined by its Tjurina algebra. It was generalized to the case of isolated varieties by Gaffney-Hauser \cite{GH85}. In the case of positive characteristic this was archieved by Greuel-Pham \cite{GrP17} and  Kerner \cite{Ker22} by considering higher Tjurina algebras, see Theorem \ref{GP18}. In this section we give a characterization of the  isomorphism type of the local ring of a hypersurface singularity $f$ without any assumption on $f$ by using higher Nash blowup algebras (Theorem \ref{main2}). Recall that a hypersurface singularity is an element in the  maximal ideal $\mathfrak{m}=(\mathbf{x})$ of the ring $R=\Bbbk[\![\mathbf{x}]\!]$ or $\Bbbk\{\mathbf{x}\}$. The {\it multiplicity} of $f$, denoted by $\operatorname{mt}(f)$, is the maximal $k$ such that $f\in\mathfrak{m}^k$. The {\it $k$-th Tjurina algebra} of $f$ is defined as
$$T_n(f)=R/_{ (f)+\mathfrak{m}^k \operatorname{j}(f)},$$
where $\operatorname{j}(f)$ denotes the Jacobian ideal of $f$. Notice that $\operatorname{j}(f)=\mathcal{J}_1(f)$ by Proposition \ref{relmatrix}. The dimension $\dim_\Bbbk T_0(f)$ is called the Tjurina number of $f$, and denoted by $\tau(f)$.
\begin{theorem}[Mather-Yau; Gaffney-Hauser; Greuel-Pham; Kerner]\label{GP18}
 Let $f, g \in \Bbbk[\![\mathbf{x}]\!]$ be such that $\operatorname{mt}(f)\geq 2$ and $\tau(f)<\infty$. Let
$$
k=
\begin{cases}
2\tau(f)-2\operatorname{mt}(f)+4, \text{ if } \mathrm{char}(\Bbbk)>0,\\
1\text{ if } \mathrm{char}(\Bbbk)=0,\\
0 \text{ if } \mathrm{char}(\Bbbk)=0  \text{ and } \Bbbk=\bar\Bbbk.
\end{cases}
$$
Then the following are equivalent:
\begin{itemize}
	\item[(i)] $f {\sim_c} g$.
	\item[(ii)] $T_n(f) \cong T_n(g)$ as $\Bbbk$-algebras for all $n\geq k$;
	\item[(iii)] $T_n(f) \cong T_n(g)$ as $\Bbbk$-algebras for some $n\geq k$.
\end{itemize} 
\end{theorem}
The following result says that for any $f$ and any $\Bbbk$, the  isomorphism type of the local ring of $f$ is determined by the $n$-th Nash blowup algebras with $n\geq 2$. Notice that, if $\Bbbk$ is an algebrically clsoed field of characteristic zero, then it can be determined by the first Nash blowup algebra \cite{MaY82,GrP17}. This is not true, in general, if the field $\Bbbk$ is either of positive characteristic or not algebraically closed as Example \ref{firstblowup} shows.
\begin{theorem}\label{main2}
Let $R=\Bbbk\{\mathbf{x}\}$ with $\mathrm{char} (\Bbbk)=0$ or $R=\Bbbk[\![\mathbf{x}]\!]$. Let $f$ and $g$ be two elements in $R$. The following are equivalent
\begin{itemize}
	\item[(i)] $f {\sim_c} g$;	
	\item[(ii)] $\mathcal T_n(f)\cong \mathcal T_n(g)$ for all $n\geq 2$;
	\item[(iii)] $\mathcal T_n(f)\cong \mathcal T_n(g)$ for some $n\geq 2$;
\end{itemize} 
\end{theorem}
For the proof of this theorem we need the following well-known results of Samuel \cite{S56}, Artin \cite{Art68} and a technical resullt (Lemma \ref{higherjacobians}).
\begin{theorem}[Samuel]\label{samuel}
Let $R=\Bbbk[\![\mathbf{x}]\!]$ be a powerseries ring over a field $\Bbbk$ and $\mathfrak{m}=\left(x_1, \ldots, x_d\right)$ be the maximal ideal of $R$. Suppose that $f$ and $g$ are elements in $R$ such that $\operatorname{j}(f) \neq R$. If $g \equiv f \bmod \mathfrak{m} \operatorname{j}(f)^2$, then there exists an automorphism $s$ of $R$ such that $s(f)=g$.
\end{theorem}
\begin{theorem}[Artin's approximation theorem]\label{artin} Let $\Bbbk$ be a valued field of characteristic zero and let $f(x, y)$ be a vector of convergent power series in two sets of variables $x$ and $y$. Assume given a formal power series solution $\widehat{y}(x)$ vanishing at 0 ,
$$
f(x, \widehat{y}(x))=0.
$$
Then there exists, for any $c \in \mathbb{N}$, a convergent power series solution $y(x)$, 
$$
f(x, \tilde{y}(x))=0,
$$
which coincides with $\widehat{y}(x)$ up to degree $c$,
$$
\tilde{y}(x) \equiv \widehat{y}(x) \text { modulo }(x)^c .
$$
\end{theorem}
\begin{lemma}\label{higherjacobians}
Let $R=\Bbbk\{\mathbf{x}\}$ or $\Bbbk[\![\mathbf{x}]\!]$  and let $\mathfrak{m}=(\mathbf{x})$. Assume that $\operatorname{mt}(f)\geq 2$ and $n\geq 2$. Then
\begin{itemize}
	\item[(i)] $\mathcal{J}_{n}(f) \subseteq \mathcal{J}_{n-1}(f)$;
\item[(ii)] $\mathcal{J}_n(f)\subseteq\mathfrak{m}\mathcal{J}_1(f)^2,$  if either $d\geq 3$ or $n \geq 3$ or $\operatorname{mt}(f)\geq 3$;
	\item[(iii)] $(f)+\mathcal{J}_{n}(f)\subseteq (f)+\mathfrak{m}\mathcal{J}_1(f)^2$.
\end{itemize} 
\end{lemma}
\begin{proof}
Part (i) follows from Proposition \ref{relmatrix} and a direct calculation. Part (ii) is due to Proposition \ref{inclusion} in case either $d\geq 3$ or $n \geq 3$. Assume that $\operatorname{mt}(f)\geq 3$, $d=n=2$. Considering $f$ as an element in $\Bbbk[\![x,y]\!]$, one can see that
$$\mathcal{J}_2(f)=\left(f_x^3,f_x^2f_y,f_xf_y^2,f_y^3, f_{xx}f_y^2-2f_{xy}f_xf_y+f_x^2f_{yy}\right)$$
respectively. Hence $\mathcal{J}_n(f)\subseteq\mathcal{J}_2(f)\subseteq\mathfrak{m}\mathcal{J}_1(f)^2$, which completes (ii).

To prove (iii) it remains to consider the case when $d\leq 2$ and $\operatorname{mt}(f)=n=2$. If $d=1$ then as $\operatorname{mt}(f)= 2$ one has $\mathcal{J}_2(f)=(f_x^2)\subseteq (f)$ and hence
$$(f)+\mathcal{J}_{2}(f)=(f)\subseteq (f)+\mathfrak{m}\mathcal{J}_1(f)^2.$$
Suppose that $d=2$ and $f\in \Bbbk[\![x,y]\!]$.  Since $\mathrm{mt}(f)= 2$, we can show that  $f$ is contact equivalent to a plane curve singularities in one of the following forms, whereas $a\in \Bbbk^*$
\begin{itemize}
	\item[(a)] $f_1=x^2,f_2=ax^2+y^2, f_3=ax^2+y^k,\ k\geq 3$ if $\mathrm{char}(\Bbbk)\neq 2$.
	\item[(b)] $f_4=xy,f_5=x^2+h(x,y),\ \mathrm{mt}(h)\geq 3$ if $\mathrm{char}(\Bbbk)=2$.
\end{itemize} 
We may compute the ideal $(f)+\mathcal{J}_2(f)$ as follows
\begin{align*}
(f_1)+\mathcal{J}_2(f_1)&=(x^2), \\
(f_2)+\mathcal{J}_2(f_2)&=(ax^2+y^2,x^3,y^3),\\
(f_3)+\mathcal{J}_2(f_3)&=(ax^2+y^k, x^3,x^2y^{k-2}),\\
(f_4)+\mathcal{J}_2(f_3)&=(x^3,xy,y^3),\\
(f_5)+\mathcal{J}_2(f_3)&=(f_5)+\mathcal{J}_2(h) \\
&\subseteq(f_5)+\mathfrak{m}\mathcal{J}_1(h)^2\text{ by (ii)}\\
&= (f_5)+\mathfrak{m}\mathcal{J}_1(f_5)^2.
\end{align*}
This implies that $(f)+\mathcal{J}_{n}(f)\subseteq (f)+\mathfrak{m}\mathcal{J}_1(f)^2$ for all $n\geq 2$.
\end{proof}
\begin{proof}[Proof of Theorem \ref{main2}]
The implication $(i)\Rightarrow (ii)$ is due to Theorem \ref{main1}, while $(ii)\Rightarrow (iii)$ is trivial. It remains to prove the implication $(iii)\Rightarrow (i)$. We first establish this implication for $R=\Bbbk[\![\mathbf{x}]\!]$. Assume that $\mathcal T_n(f)\cong \mathcal T_n(g)$ for some $n\geq 2$. Then by the lifting lemma (see for instance \cite[Lemma I.1.23]{GLS06}, there exists an automorphism $\phi$ of $R$ such that $$(g)+\mathcal{J}_n(g)=\phi \left((f)+\mathcal{J}_n(f)\right)=(\phi(f))+\mathcal{J}_n(\phi(f)),$$
where the later identity is according to Lemma \ref{local-lem1}. Replacing $f$ by $\phi(f)$ we may assume that\\ $(f)+\mathcal{J}_n(f)=(g)+\mathcal{J}_n(g).$ We may assume, without loss of generality, that $\mathrm{mt}(f)\geq \mathrm{mt}(g)$. Considering the first case that $\mathrm{mt}(f)=1$, one obtains $\mathrm{mt}(g)=1$ and hence $f\sim_c g$. Assume that $\mathrm{mt}(f)\geq 2$, then by Lemma \ref{higherjacobians} one has $(f)+\JI_n(f)\subseteq (f)+\mathfrak{m}\mathcal{J}_1(f)^2$, where $\mathfrak{m}=(\mathbf{x})$ is the maximal ideal of $R$. Since  
$$g\in (f)+\mathcal{J}_n(f)\subseteq (f)+\mathfrak{m}\mathcal{J}_1(f)^2= (f)+\mathfrak{m}\operatorname{j}(f)^2,$$ there exist $u,h\in R$ such that
$$g=u f +h, h\in \mathfrak{m}\operatorname{j}(f)^2.$$
As $h\in  \mathfrak{m}\operatorname{j}(f)^2$, it follows that
$$\mathrm{mt}(h)\geq 1+2(\mathrm{mt}(f)-1)>\mathrm{mt}(f).$$
Then $u$ must be a unit in $R$ since $\mathrm{mt}(g)\leq \mathrm{mt}(f)<\mathrm{mt}(h)$. It yields that 
$$u^{-1}g-f=u^{-1}h\in \mathfrak{m}\operatorname{j}(f)^2.$$
Applying Samuel's theorem (Theorem \ref{samuel}) one obtains an automorphism $s$ of $R$ such that 
\begin{align}\label{formal}
s(f)=u^{-1}g,
\end{align}
which gives $f\sim_c g$.

We now consider the case that $R=\Bbbk\{\mathbf{x}\}$ with $\mathrm{char}(\Bbbk)=0$. Let $F({\bf x,y})$ be a convergent series in $\Bbbk\{\mathbf{x,y}\}$ defined as
$$F({\bf x,y})=y_{d+1}f(y_1,\ldots,y_d)-g(x_1,\ldots,x_d).$$
By (\ref{formal}), $F$ admits a formal solution $\hat y\in \Bbbk[\![\mathbf{x}]\!]^{d+1}$ with 
$$\hat y_i=s(x_i), \forall i=1,\ldots,d\text{ and } \hat y_{d+1}=u(x).$$
Applying Artin's approximation theorem (Theorem \ref{artin}) we obtain a convergent solution $\tilde{y}\in \Bbbk\{\mathbf{x}\}^{d+1}$ of $F$ such that 
$$\hat y_i-\tilde{y}_i\in \mathfrak{m}^2, \forall i.=1,\ldots, d+1$$
Then $\tilde{y}$ induces an automorphism $\tilde{s}$ of $\Bbbk\{\mathbf{x}\}$ defined as $\tilde{s}(x_i)=\tilde{y}_i$, for all $i=1,\ldots,d$ and $\tilde y_{d+1}$ is a unit in $\Bbbk\{\mathbf{x}\}$. Hence $ \tilde{y}_{d+1} \tilde s(f)=g$ and therefore $f\sim_c g$.
\end{proof}
\begin{example}[Mather-Yau, Gaffney-Hauser]\label{firstblowup}

\begin{enumerate}[(a)]
\item Let $p=\mathrm{char}(\Bbbk)>0$. Take $f(x,y)=x^{p+1}+ y^{p+1}$ and $g=f+x^{p}$, then 
$$f_x=g_x=x^{p}, f_y=g_y=y^{p}$$
and therefore $\mathcal T_1(f)=\mathcal T_1(g)$ but $f \not\sim_c g$.
\item  Let $\Bbbk=\Bbb R$. Take $f(x,y)=x^{2}+ y^{2}$ and $g=x^{2}- y^{2}$ in $\Bbb R[\![x,y]\!]$. Then $\operatorname{j}(f)=\operatorname{j}(g)=(x,y)$ and therefore 
$$\mathcal T_1(f)=\mathcal T_1(g)\cong \Bbb R,$$ but it is obvious that $f \not\sim_c g$.
\end{enumerate}
\end{example}

\subsection*{Acknowledgment.}  A part of this work was done while the author were visiting  Vietnam Institute for Advanced Study in Mathematics (VIASM) and Hanoi Institute of Mathematics, VAST in 2024. The author would like to thank these institutes for hospitality and support.



\end{document}